\documentclass[a4paper,12pt,reqno]{amsart}
\usepackage{mathrsfs,amsmath,amssymb,latexsym,amsfonts, amscd}
\usepackage{showlabels}

\title{On a question of Demailly-Peternell-Schneider}
\author{Meng Chen and Qi Zhang}

\dedicatory
{Dedicated to the memory of Eckart Viehweg}

\address{Department of Mathematics \& LMNS, Fudan University, Shanghai 200433, China}
\email{mchen@fudan.edu.cn}

\address{Department of Mathematics, University of Missouri, Columbia, MO
65211, USA}
\email{qi@math.missouri.edu}

\thanks{The first author was supported by National Natural Science Foundation of China (\# 11171068), NSFC for Innovative Research Groups (\# 11121101) and Doctoral Fund of Ministry of Education of China (\# 20110071110003)}

\newcommand{\bQ}{{\mathbb Q}}
\newcommand{\bP}{{\mathbb P}}

\newcommand\lrw{\longrightarrow}
\newcommand\rw{\rightarrow}

\newcommand{\KX}{K_X}
\newcommand{\KY}{K_Y}

\newcommand{\OXp}{{\mathcal O}_{X'}}
\newcommand{\OY}{{\mathcal O}_Y}

\newtheorem{thm}{Theorem}[section]
\newtheorem{lem}[thm]{Lemma}
\newtheorem{cor}[thm]{Corollary}

\theoremstyle{definition}
\newtheorem{defn}[thm]{Definition}

\newtheorem{exmp}[thm]{Example}

\newtheorem{rem}[thm]{Remark}
\theoremstyle{remark}

\begin{document}
\begin{abstract}  We  give an affirmative answer to an open question posed by Demailly-Peternell-Schneider \cite{DPS} in 2001
and recently by Peternell  in \cite{P}. Let $f:X\rw Y$ be a surjective morphism from a log canonical pair $(X,D)$ onto a $\bQ$-Gorenstein variety $Y$. If $-(K_X+D)$ is nef, we show that $-K_Y$ is pseudo-effective.\end{abstract}
\maketitle

\pagestyle{myheadings}
\markboth{\hfill M. Chen and Q. Zhang\hfill}{\hfill  On a question of Demailly-Peternell-Schneider \hfill}
\numberwithin{equation}{section}
\section{\bf Introduction}

Projective varieties with numerically effective  anticanonical bundles appear naturally in the Minimal Model Program. In the surface case, these objects include del-Pezzo surfaces, K3 surfaces, Abelian
surfaces and so on. The higher dimensional case has been investigated intensively by several authors, for instance,
Campana \cite{C}, Demailly-Peternell-Schneider \cite{DPS}, Peternell \cite{P} and the second author \cite{Z2}. A key step in understanding
the geometry of these varieties is determining projective morphisms between
them. Let $f: X\rw Y$ be a surjective morphism with $-\KX$ nef. What can one say about $-\KY$? In this direction,
one of the main open questions was asked by Demailly-Peternell-Schneider \cite{DPS} and recently by Peternell \cite{P}:
\medskip

\noindent{\bf Question A}. {\em  Let $X$ be a complex projective manifold on which $-\KX$ is nef. Let $f:X\rw Y$ be a projective morphism onto the projective manifold $Y$.
Is $-\KY$ pseudo-effective?}
\medskip

Throughout we work over the complex number field ${\mathbb C}$.

The main result in  this paper is an affirmative answer to Question A. We  prove the following stronger theorem:
\medskip

\noindent{\bf Main Theorem}. {\em Let $X$ be a normal projective variety and $D$  an effective $\bQ$-divisor on $X$ such that
the pair $(X,D)$ is log canonical. Let $Y$ be a normal and $\bQ$-Gorenstein projective variety. If $f: X\rw Y$ is a surjective
morphism and $-(K_X+D)$ is nef,
then $-K_Y$ is pseudo-effective.}
\medskip

\begin{rem}  Naively, under the same condition as in the main theorem, one might expect that $-\KY$ is nef. Unfortunately,
this is false in general.
\end{rem}

The following example was kindly provided to us by Y. Prokhorov:

\begin{exmp} Embed ${\Bbb P}^1\times {\Bbb P}^1$ into ${\Bbb P}^5$ through the linear system $|(1,2)|$ to obtain a
projective cone $W\subset {\Bbb P}^6$ over the surface $S\cong{\Bbb P}^1\times {\Bbb P}^1$ with the normal sheaf ${\mathcal N}_{S/W}\cong {\mathcal O}_S(-1,-2)$
and  vertex $Q=[0,0,0,0,0,0,1]$. Let $\pi: X\lrw W$ be the blow-up at the vertex $Q\in W\subset {\Bbb P}^6$.
Then $X$ is a nonsingular ${\Bbb P}^1$-bundle over $S$ with the exceptional divisor $E\cong S$. Since $K_X=\pi^*(K_W)+E$,
$-K_X=-\pi^*(K_W)-E$ is nef, one of the curve families on $S$ is $\KX$-trivial, while another curve family $\{l\}$ on $S$ is $K_X$-negative.
Mori's $K$-negative contraction $f=\text{Cont}_{l}:X\rightarrow Y$ maps $S$ onto a curve $C=f(S)\cong {\Bbb P}^1$ on $Y$.
By Mori's theorem \cite{M}, $Y$ is nonsingular. According to Proposition 4.5 and Corollary 4.6 of Prokhorov \cite{Pr}, it follows that
$K_Y\cdot C=2>0$.
Thus $-K_Y$ is not nef.
\end{exmp}

In the above example, $f$ is birational. Recently, we learned from Fujino and Gongyo's paper \cite{FG} that H. Sato
had constructed the following example where $f: X \rw Y$ is flat:

\begin{exmp} Let $\Sigma$ be the fan in ${\Bbb R}^3$ whose rays are generated by $x_1=(1:0:1)$, $x_2=(0:1:0)$,
$x_3=(-1:3:0)$, $x_4=(0:-1:0)$, $y_1=(0:0:1)$ and $y_2=(0:0:-1)$. Their maximal cones are $<x_1,x_2,y_1>$, $<x_1,x_2,y_2>$, $<x_2,x_3,y_1>$, $<x_2,x_3,y_2>$,
$<x_3,x_4,y_1>$, $<x_3,x_4,y_2>$, $<x_4,x_1,y_1>$ and $<x_4,x_1,y_2>$. One can see that $X(\Sigma)$,  the toric threefold corresponding to the fan $\Sigma$,  has a
${\Bbb P}^1$ fibration $f$ onto $Y={\Bbb P}_{\bP1}({\mathcal  O}_{{\Bbb P}^1}\oplus {\mathcal O}_{{\Bbb P}^1}(3))$
(see \cite{FG} for more details). In fact,   $X$ is a weak Fano threefold, but $-K_Y$ is not nef.
\end{exmp}

\begin{rem} If we replace the condition that $-(K_X+D)$ is nef by $-(K_X+D)$ is pseudo-effective, then
the conclusion is again false. This may explain the motivation behind  Question A.\end{rem}

The following example  is found in \cite{Z1}:

\begin{exmp} Let $C$ be a smooth curve of
genus $g$ and $A$ an ample line bundle on $C$ such that
$\deg A>2\deg K_{C}$. Let $S=\text{Proj}_C({\mathcal
O}_C(A)\oplus {\mathcal
O}_{C})$. Then
$K_{S}={\pi}^*(K_{C}+A)-2L$, where $L$ is the
tautological bundle. Clearly  $-K_S$ is big (since $h^{0}(S,-mK_S)\approx c\cdot m^{2}$ for $m\gg 0$ and some constant
$c>0$). Thus $-K_S$ is pseudo-effective. In this setting, $S$ is a $\bP^1$-bundle over $C$, but $-K_C$ is never pseudo-effective if $g>1$.
{\it This example suggests the necessity of the condition  $(X,D)$ is log canonical. Indeed,  if we  let   
$D=-K_S=2(L-{\pi}^{*}A)+{\pi}^{*}(A-K_C)$,
the pair $(S,D)$ is not log canonical.}
\end{exmp}

\section{\bf Proof of the Main Theorem}

Before we  prove our Main Theorem, let us
recall some related definitions and known
results. For general references on the minimal model
program, see (for instance) \cite{K-M, KMM}.

\begin{defn} (Viehweg \cite{V})  Let $X$ be a smooth (quasi-)projective
variety and  ${\mathcal F}$ a torsion-free
coherent sheaf on $X$. We say ${\mathcal F}$ is {\it weakly positive on $X$},
if there exists some Zariski open subvariety $U\subset X$ such that for every ample invertible sheaf ${\mathcal H}$
and every positive integer $m$, there exists some positive integer $n$ such that $
{\hat {\mathcal S}}^{mn}(\mathcal F)\otimes \mathcal H^{n}$ is generated by global sections over $U$,
where ${\hat {\mathcal S}^{n}}$ denotes
the reflexive hull of ${\mathcal S}^{n}$.
\end{defn}

In the proof of our theorem, we need the following weak positivity statement that was originally developed by
Viehweg \cite{V} and, later on,  generalized by Campana \cite{C} and Lu \cite{L}.

\begin{lem} Let $V$ be a smooth projective variety
and $M$
an effective $\bQ$-divisor on $V$ such that the pair $(V,M)$ is  log canonical.
If  $\varphi: V\rightarrow W$ is a surjective morphism onto the nonsingular variety $W$,
then $\varphi_{*}m(K_{V/W}+M)$ is torsion free and
weakly positive for every $m\in {\mathbb Z}_{>0}$ such that $m(K_{V/W}+M)$ is Cartier.
\end{lem}
\bigskip

\begin{proof}[{\bf Proof of the Main Theorem}]
 Let $p: Y'\rw Y$  be a resolution of $Y$.   Let $\pi: X'\rw X$ be a log resolution of
$(X, D)$, such that the induced rational map $X'\dashrightarrow Y'$ is in fact a morphism $g': X'\rw Y'$.
  We have  the following commutative diagram:
$$
\CD
X'    @>\pi>> X  @. \\
@Vg'VV   @VVfV   @. \\
Y'    @>p>>    Y @.
\endCD
$$
Since the pair $(X, D)$ is log canonical, we have  $$ K_{X'}={\pi}^{*}(K_{X}+D)+\sum d_{i}E_{i},$$ where
$d_{i}\geq -1$ for each $i$ and $\sum E_{i}$ is a divisor  with normal crossings only.  We may also assume that the support of ${f'}^{-1}(\text{Sing}(Y))$ is a divisor with normal crossings only, where $f'=f\circ \pi$.  Set $Y^0:=Y-\text{Sing}(Y)$.

Pick an ample divisor $L$ on $Y$ and a rational number $\delta>0$. Our goal is to show that $-K_Y+2\delta L$ is equivalent to an effective $\bQ$-divisor, which implies that $-K_Y$ is pseudo-effective.

Let $A$ be a sufficiently very ample divisor on $X$. By Serre's theorem, for every rational number $\varepsilon>0$ and every divisible integer
$m>0$ such that ${\varepsilon}/{\delta}$ is sufficiently small and  $m\delta$ and $m\varepsilon$ are integers, the sheaf
$${\mathcal O}_Y(m\delta L)\otimes \big(f'_*{\mathcal O}_{X'}(m\varepsilon \pi^*(A))\big)^*$$
is generated by global sections. Thus, for all integers $n>0$, it follows that
\begin{eqnarray*}
&&{\mathcal O}_Y(nm\delta L)\otimes \bigotimes^n \big(f'_*{\mathcal O}_{X'}(m\varepsilon \pi^*(A))\big)^*\ \text{and}\\
&& {\mathcal O}_Y(nm\delta L)\otimes S^n \big((f'_*{\mathcal O}_{X'}(m\varepsilon \pi^*(A)))^*\big)
\end{eqnarray*}
are both generated by global sections.

Since $-(K_X+D)$ is nef by assumption, $-(K_X+D)+\varepsilon A$ is $\bQ$-ample. So we may  assume
that $-{\pi}^{*}(K_X+D)+\varepsilon \pi^*(A)$ is $\bQ$-equivalent to an effective $\bQ$-divisor. We may also assume that $$\text{Supp}(-{\pi}^{*}(K_X+D)+\varepsilon \pi^*(A)+\sum E_i)$$ is a divisor with  normal crossings only.
Let
$\Delta =-{\pi}^{*}(K_X+D)+\varepsilon \pi^*(A)+\sum {a}_{i}E_{i}$, where $a_i=-d_i$ whenever $d_{i}<0$ and $a_i=0$ otherwise.

We may write
$$K_{X'/Y}+\Delta \sim_{\bQ}
\varepsilon \pi^*(A)-{f'}^{*}K_{Y}+\hat{\Delta}$$ where $\hat{\Delta}:=\sum_{j\in J}d_jE_j$, $d_j\geq 0$
and $E_j$ is $\pi$-exceptional for each $j\in J$. Clearly the pair $(X',\Delta)$ is log canonical.

Choose a divisible sufficiently large integer $m>0$ such that $m\varepsilon A$, $m\delta L$, $m d_jE_j$
and $-mK_Y$ are all Cartier and define
$$\omega:=m(\varepsilon \pi^*(A)-{f'}^{*}K_{Y}+\hat{\Delta}).$$
By Lemma 2.2, there is  a Zariski open smooth subvariety $U\subset Y$ with $\text{codim}(Y-U)\geq 2$ such that
$$f'_{*}\omega={f'}_{*}{\mathcal O}_{X'}(m(\varepsilon \pi^*(A)+\hat{\Delta}))\otimes {\mathcal O}_Y(-mK_{Y})$$ is locally free and weakly positive over $U$.
Recall that  $\delta L$ is an ample $\bQ$-divisor on $Y$. Since each $E_j$ ($j\in J$) is exceptional with respect to $\pi$,
the effect of $\hat{\Delta}$ can be neglected when   ${f'}_{*}\omega$ is restricted  to $U$. Thus
$${f'}_{*}{\mathcal O}_{X'}(m\varepsilon \pi^*(A))|_U={f'}_*\omega|_U\otimes {\mathcal O}_Y(mK_Y)|_U$$ is locally free over $U$. Furthermore,
weak positivity implies that{\tiny
\begin{eqnarray*}
&&({\mathcal S}^{n}{f'}_{*}\omega)|_U \otimes {\mathcal O}_Y(nm\delta L)|_U\\
&=&{\mathcal S}
^{n}({f'}_{*}{\mathcal O}_{X'}(m\varepsilon \pi^*(A)+\sum md_jE_j))|_U\otimes {\mathcal O}_Y(-nmK_Y)|_U\otimes {\mathcal O}_Y(nm\delta L)|_U\\
&=&{\mathcal S}
^{n}({f'}_{*}{\mathcal O}_{X'}(m\varepsilon \pi^*(A)))|_U\otimes {\mathcal O}_Y(-nmK_Y)|_U\otimes {\mathcal O}_Y(nm\delta L)|_U
\end{eqnarray*}
}is generated by global sections at the generic point of $U$ for some sufficiently large  $n$. Since there is a functorial isomorphism between
$S^n({\mathcal F}^*)$ and $(S^n{\mathcal F})^*$ for locally
free sheaves, we have a non-trivial trace map:
$${{\mathcal S}^{n}\big(({f'}_{*}{\OXp}(m\varepsilon \pi^*(A)))^{*}\big)|_U\otimes {\mathcal S}^{n}({f'}_{*}{\OXp}(m\varepsilon \pi^*(A)))}|_U\rightarrow {\mathcal O}_U.$$
As we have seen, $$\OY(nm\delta L)|_U\otimes {\mathcal S}^{n}\big(({f'}_{*}{\OXp}(m\varepsilon \pi^*(A)))^{*}\big)|_U$$
is generated by global sections over $U$. Thus{\tiny
\begin{eqnarray*}
{\mathcal E}_U&:=&\OY(nm\delta L)|_U\otimes {\mathcal S}^{n}\big(({f'}_{*}{\OXp}(m\epsilon \pi^*(A)))^{*}\big)|_U\otimes
({\mathcal S}^{n}{f'}_{*}\omega)|_U \otimes \OY(nm\delta L)|_U
\\
&=&\OY(-nm\KY+2nm\delta L)|_U\otimes \big({{\mathcal S}^{n}\big(({f'}_{*}{\OXp}(m\varepsilon \pi^*(A)))^{*}\big)|_U\otimes {\mathcal S}^{n}({f'}_{*}{\OXp}(m\varepsilon \pi^*(A)))}|_U\big)
\end{eqnarray*}
}is generated by global sections at the generic point of $U$. There are induced  maps
$$\bigoplus {\mathcal O}_U\overset{\alpha}\lrw {\mathcal E}_U\overset{\beta}\lrw \OY(-nm\KY+2nm\delta L)|_U$$
such that $\alpha$  is surjective and $\beta$
is non-zero.
This proves that
$$H^0(U,\OY(-nmK_{Y}+2nm\delta L)|_U)\neq 0.$$
Because $\text{codim}(Y-U)\geq 2$, the divisor $n(-mK_{Y}+2m\delta L)$ is effective on $Y$.
\end{proof}

As instant applications, we can improve some earlier results in \cite{Z1, Z2}.

\begin{cor} Let $X$ be a normal projective
variety and $D$ an effective $\bQ$-divisor on $X$ such
that the pair $(X,D)$
is log canonical and that $-(K_X+D)$ is nef.
Let $f: X\rightarrow Y$ be a surjective morphism onto $Y$ which is a $\bQ$-Gorenstein projective variety. Then either
\begin{itemize}
\item[(1)]  $Y$ is uniruled; or
\item[(2)] $K_Y\sim_{\bQ} 0$.
\end{itemize}
\end{cor}
\begin{proof} If Y is not uniruled, then $K_Y$ is pseudo-effective by \cite{BDPP}. However our Main
Theorem says $-K_Y$ is pseudo-effective, which means $K_Y\sim_{\bQ} 0$.
\end{proof}

\begin{rem}  The weaker statement corresponding to (2) in \cite{Z2} is the Kodaira
dimension $\kappa (Y)=0$. 
\end{rem}

\begin{cor} (cf. \cite{Z1}) Let $X$ be a smooth projective variety with nef $-K_X$. Then the
Albanese map $\text{Alb}_{X}:X\rightarrow
\text{Alb}(X)$ is surjective and has connected fibers
\end{cor}
\begin{proof}  Let $a:X\rw V$ be the induced morphism with connected fibers after the Stein factorization of  $\text{Alb}_X$.
We may even assume that $V$ is smooth. Then $\kappa(V)=0$ by the Main Theorem, and Kawamata's theorem in \cite{K} implies that $V$ is an abelian variety. Thus it follows that $V=\text{Alb}(X)$ due to the universal property of $\text{Alb}_X$.
\end{proof}
\medskip

{\bf Acknowledgment.}  Chen  thanks The Max-Planck-Institut fuer Mathematik (Bonn) for the invitation and the generous
support in 2011. Zhang thanks the Department of Mathematics and the Key Laboratory of Mathematics for Nonlinear
Sciences (MOE), Fudan University for the hospitality and support during his visit in the summer of 2011 when part of the paper was written.

\end{document}